\documentclass[a4paper, 12pt]{amsart}
\usepackage{amsfonts,amssymb,amsmath,amscd}
\usepackage[margin=2.6 cm]{geometry}

\usepackage{ytableau}
\usepackage{nicematrix}
\usepackage{tikz-cd}
\usepackage{xcolor}
\usepackage{enumerate}
\usepackage{multirow}
\usepackage{float}
\usepackage{mathtools}
\usepackage{mathrsfs}

\usepackage[colorlinks=true,citecolor=blue,backref,pagebackref,urlcolor=blue,linkcolor=blue]{hyperref}

\definecolor{fgreen}{RGB}{44,144, 14}

\renewenvironment{proof}{{\bfseries Proof.}}{\qed}

\numberwithin{equation}{section} 

\newtheorem{theorem}{Theorem}[section] 
\newtheorem{proposition}[theorem]{Proposition} 
\newtheorem{corollary}[theorem]{Corollary} 
\newtheorem{lemma}[theorem]{Lemma} 

\theoremstyle{definition}
\newtheorem{definition}[theorem]{Definition} 
\newtheorem{remark}[theorem]{Remark}

\everymath{\displaystyle}
\def\R{\mathbb R}

\def\C{\mathbb C}

\def\R{\mathbb R}

\newcommand{\st}{\mathrm{Stab}}

\def\M{{\mathrm M}}

\newcommand{\thmref}[1]{Theorem~\ref{#1}}
\newcommand{\lemref}[1]{Lemma~\ref{#1}}

\newcommand{\propref}[1]{Proposition~\ref{#1}}
\newcommand{\corref}[1]{Corollary~\ref{#1}}

\makeatletter
\@namedef{subjclassname@2020}{\textup{2020} Mathematics Subject Classification}
\makeatother

\begin{document}

\title[Functorial Free Group from Anosov Representations]{Functorial Free Group from Anosov Representations on Bundles}
\author[K. Gongopadhyay]{Krishnendu Gongopadhyay}
\author[T. Nayak]{Tathagata Nayak}

\address{Indian Institute of Science Education and Research (IISER) Mohali,
	Knowledge City,  Sector 81, S.A.S. Nagar 140306, Punjab, India}
\email{krishnendu@iisermohali.ac.in}

\address{Indian Institute of Science Education and Research (IISER) Mohali,
	Knowledge City,  Sector 81, S.A.S. Nagar 140306, Punjab, India}
 \email{tathagatanayak68@gmail.com}


\date{December 1, 2025}
\begin{abstract}
Let $\rho: \Gamma \to G$ be an Anosov representation, with $\Gamma$ a word hyperbolic group and $G$ a semisimple Lie group. Previous works (Guichard–Wienhard, Kapovich–Leeb–Porti, and Carvajales–Stecker) constructed an open domain of discontinuity $\Omega_\rho \subset G/H$, where $H$ is a parabolic or symmetric subgroup. In this paper, we extend the properly discontinuous $\Gamma$-action via $\rho$ to the space of connections on the pullbacks of the tangent bundle over $\Omega_\rho$.  

When $\Omega_\rho$ is a complex curve, we show that the $\Gamma$-action is properly discontinuous on the union of Higgs bundle structures associated with the $(1,0)$ part of the complexified pullback bundles.  We further construct a  free abelian group $F^{ab}$ generated by these holomorphic line bundles and induce a topoogical structure on it, so that  $\rho(\Gamma)$ acts properly discontinuously on $F^{ab} \setminus \{\mathrm{id}\}$. This free abelian group is well-defined up to isomorphism over the character variety of Zariski dense Anosov representations. Finally, we endow the space of Anosov representations with a categorical structure compatible with $\Omega_\rho$ and construct a natural functor to the category of abelian groups.

\end{abstract}

%\tableofcontents
	
	\maketitle 

    \section{Introduction}
Anosov representations, introduced by Labourie \cite{La}, have emerged as central objects in higher Teichmüller theory, lying at the intersection of geometric group theory, dynamics, and low-dimensional topology. These representations generalize convex cocompactness to higher-rank Lie groups and exhibit rich topological and geometric behavior, including the existence of limit maps and domains of discontinuity in associated flag varieties. They have been extensively studied and applied in diverse geometric and dynamical contexts; see, e.g., \cite{gw, ka0, ka, kl, os, po, wi}.

A key feature of Anosov representations is the existence of domains of discontinuity in flag varieties $\mathcal{X} = G/P$, where $P$ is a parabolic subgroup. These domains, introduced by Guichard–Wienhard \cite{gw} and Kapovich–Leeb–Porti \cite{klp}, provide natural geometric settings for dynamics and compactifications. More recently, Carvajales and Stecker \cite{cs} extended this construction to general $G$-homogeneous spaces $\mathcal{X}$, under the assumption that the $G$-action on $G/P \times \mathcal{X}$ has finitely many orbits, which holds when $\mathcal{X} = G/H$ with $H$ symmetric or parabolic. We denote either of these discontinuity domains by $M := \Omega_\rho$. 

 Let,  $(P^+,P^-)$ be a pair of opposite parabolic subgroups of a semisimple Lie group $G$,  cf. \cite{gw}. For a word hyperbolic group $\Gamma$,  let $\rho: \Gamma \to G$ be a $(P^+,P^-)$-Anosov representation.  For each $g \in \rho(\Gamma)$ we consider the pullback of the tangent bundle
over the domain of discontinuity $\Omega_\rho$ (denoted $M$) and write 
$E^g := g^{*}TM$.
When $M$ is a complex manifold we also consider the pullback of the
$(1,0)$--part of the complexified tangent bundle $TM \otimes_{\mathbb{R}} \mathbb{C}$
and denote these bundles by
$ \mathcal{E}^g := g^{*}T^{1,0}M$. 

The pullback bundles $E^g := g^* TM$ (with $\mathcal{E}^g$  in the complex case) are abstractly isomorphic to $TM$ (respectively $T^{1,0}M$) but carry the geometric structures pulled back by $g$, encoding the dynamics of $\rho(\Gamma)$ on $\Omega_\rho$.  We show that proper discontinuity of $\rho(\Gamma)$ on $\Omega_\rho$ extends to the associated geometric structures, including pullbacks of the tangent bundle, holomorphic line bundles, connections, and Higgs bundles.

When $\Omega_\rho$ is a complex curve, we construct a topological free group $F$ generated by these bundles and their duals, equipped with a topology induced from $\Omega_\rho$ and invariant under $\rho(\Gamma)$. The group $\rho(\Gamma)$ acts properly discontinuously on $F \setminus \{\mathrm{id}\}$, providing an extension of the classical Picard group that encodes the equivariant structure of the bundles.

We now state the main results. We prove that $\Gamma$ acts properly discontinuously on the union of spaces of $k$-forms with coefficients in their endomorphism bundles (Proposition~\ref{2.3}) and extend this to connections.
\begin{theorem}\label{2.4}
    $\rho(\Gamma)$ acts properly discontinuously on the union of  the connections of $\{E^g\}_{g\in{\rho(\Gamma)}}$.
\end{theorem}
As a corollary, the action of $\rho(\Gamma)$ is properly discontinuous on
$$
\coprod_{g \in \rho(\Gamma)} \{\text{connections on } E^g\}\big/{\mathcal{A}ut(E^g)},
$$
where $\mathcal{A}ut(E^g)$ denotes the group of gauge automorphisms of $E^g$, see \corref{cor1}.

\medskip When $\Omega_\rho$ is equipped with a complex structure, we consider the holomorphic vector bundles obtained by pulling back the $(1,0)$-part of the complexified tangent bundle. Let $\Omega^{p,q}(M; \mathrm{End}(\mathcal{E}^g))$ denote the space of $(p,q)$-forms with coefficients in $\mathrm{End}(\mathcal{E}^g)$, for each $g \in \rho(\Gamma)$. In this setting, we show the following. 

\begin{theorem} \label{thm 3.4}
$\rho(\Gamma)$ acts properly discontinuously on 
$$
\coprod_{g \in \rho(\Gamma)} \Omega^{p,q}(M; \mathrm{End}(\mathcal{E}^g)), \quad p,q \in \{0,1,\dots, \dim_\C M\}.
$$ 
Moreover, the action is properly discontinuous on the union of the spaces of pseudo-connections of $\{\mathcal{E}^g\}_{g \in \rho(\Gamma)}$.
\end{theorem}

\subsection*{For the remainder of the paper, we focus on the case where \texorpdfstring{$\Omega_\rho$}{Omega	extunderscore rho} is a complex curve} The associated holomorphic line bundles $\{\mathcal{E}^g\}_{g \in \rho(\Gamma)}$ naturally admit Higgs bundle structures. We extend the group action to these Higgs bundles and establish the following results.

\begin{theorem}\label{thm1.1}
    The action of $\Gamma$ by $\rho$ is properly discontinuous on the union of Higgs bundle structures of $\{\mathcal{E}^g\}_{g\in{\rho(\Gamma)}}$, when $\Omega_\rho$ is a complex curve.
\end{theorem}
As a consequence, it can be said that, the action of $\rho(\Gamma)$ is properly discontinuous on $$\coprod_{g\in{\rho(\Gamma)}}{\{\hbox{Higgs bundle structures on}\;\mathcal{E}^g\}\big/ {\mathcal{A}ut(\mathcal{E}^g)}}.$$

To organize the collection of these holomorphic line bundles, we define a free group $F$ generated by   pullback bundles corresponding to  equivalence classes in $\rho(\Gamma)$, where two elements $g,~g' \in \rho(\Gamma)$ are considered equivalent if they  induce the same biholomorphism on the domain of discontinuity $\Omega_\rho$. This group, with identity given by the trivial line bundle (denoted by $\mathcal{O}$) and inverses corresponding to dual bundles, has a topological structure and admits a natural $\rho(\Gamma)$-action by group automorphisms. 

Let $F^{ab}$ be the abelianization of this free group. Note that, the topology on $F$ is the coarse initial topology induced by the first–letter projection $\pi\colon F \setminus \{\mathcal{O}\} \to M$. In contrast, the topology of $F^{ab}$ is modeled on the disjoint union $\bigsqcup_{n\in \mathbb{N}} \overline{M}^{(n)}$,  where each $\overline{M}^{(n)}$ is obtained as the Hausdorff quotient $M^{(n)}/S_n$ for  $M^{(n)}$ to be  the direct product of $n$-copies of $M$.
We observe the following. A version of the theorem remains valid if we replace  $F^{ab}$ by $F$. 

\begin{theorem}\label{thm5.2}
  $\rho(\Gamma)$ acts as a group of automorphisms on the free abelian group $F^{ab}$.   Moreover, $F^{ab}$ has a structure of non-trivial topological space so that the action of $\rho(\Gamma)$ is properly discontinuous on $F^{ab} \setminus \{\mathcal{O}\}$.
\end{theorem}

We then study how this free abelian group construction behaves on the character variety of Zariski dense Anosov representations. In particular, we prove that the isomorphism class of $F^{ab}$ is invariant under the conjugation action of $G$ on $\rho$, thus providing a well-defined assignment of such groups to points in the character variety (\thmref{t2.11}). Exploring this further could reveal new rigidity or deformation phenomena associated with Anosov representations.

The appearance of a category of Anosov representations is motivated by the functorial nature of our construction. To each Anosov representation $\rho : \Gamma \to G$, we associate a family of holomorphic line bundles $\{\mathcal{E}^g_\rho\}_{g \in \rho(\Gamma)}$ on the domain of discontinuity $\Omega_\rho$. If $\rho$ is replaced by another Anosov representation $\rho'$, and if $\varphi : \rho \to \rho'$ is a morphism intertwining the group actions, then the bundles $\{\mathcal{E}^g_\rho\}$ correspond naturally to the bundles $\{\mathcal{E}^{\varphi(g)}_{\rho'}\}$. To capture this naturality precisely, it is convenient to view Anosov representations as objects of a category and to interpret the assignment  
$$
\rho \longmapsto F_{\rho}^{ab}
$$  
as a functor. 

We introduce a categorical structure on the space of $(P^+,P^-)$-Anosov representations. Fixing a type of domain of discontinuity, we define morphisms between representations compatible with the geometric structures on $\Omega_\rho$. This leads to a functor  
$$
\mathcal{F}: \mathbf{An} \longrightarrow \mathbf{Ab}
$$  
from this category to the category of abelian groups, sending each representation to its associated free abelian group. This yields the following theorem.

\begin{theorem}\label{thm1.3}
    The space of $(P^+,P^-)$-Anosov representations forms a category. The map $\mathcal{F}$, associating to each representation its corresponding free abelian group, is a covariant functor to the category of abelian groups.
\end{theorem}

\begin{center}
  \begin{tikzcd}[row sep=huge]
\textbf{An} \arrow[r, dotted,"\mathcal{F}"] & \textbf{Ab}
  \end{tikzcd}
\end{center}

\begin{center}
  \begin{tikzcd}[row sep=huge]
\rho \arrow[r, dotted,"\mathcal{F}"] \arrow[d, "\phi" ] & F_{\rho}^{ab} \arrow[d, "\mathcal{F}(\phi)" ] \\ 
\rho' \arrow[r, dotted, "\mathcal{F}" ] &  F_{\rho'}^{ab}
\end{tikzcd}
\end{center}

\medskip 

%It would be curious to see whether   there exists a connected  component in the character variety on which the map that associates each element of the character variety to its isomorphism class of free groups, is injective. This direction could reveal new rigidity or deformation phenomena associated with Anosov representations.
The construction of the free abelian group $F_{\rho}^{ab}$ in \S\ref{ab} depends on the choice of the domain of discontinuity $\Omega_\rho$. A similar functorial correspondence also holds from $\textbf{An}$ to the category of free groups $\textbf{Fr}$. We refer to \S\ref{5.2} for the construction of the free group that corresponds to a $(P^+, P^-)$-Anosov representation. 

In the context of \S \ref{subsec2.6}, little is known about the relationship between the categorical structures of Anosov representations associated with different domains, say $\Omega_\rho$ and $\widetilde{\Omega}_\rho$. Let their associated categories be $\mathbf{An}$ and $\widetilde{\mathbf{An}}$, respectively. One may ask under what conditions these categories are isomorphic. It is also natural to investigate the existence of a group homomorphism 
$$
\eta_\rho: F_{\rho}^{ab} \longrightarrow \widetilde{F}_{\rho}^{ab}
$$ 
for each Anosov representation $\rho$, such that for every morphism 
\(\phi \in \mathrm{Hom}_{\mathbf{An}}(\rho, \rho') \cap \mathrm{Hom}_{\widetilde{\mathbf{An}}}(\rho, \rho')\), the following diagram commutes:
$$
\mathcal{F}(\phi)\circ \eta_{\rho'} = \eta_\rho \circ \widetilde{\mathcal{F}}(\phi).
$$
Establishing such a correspondence would demonstrate the consistency of categorical and geometric structures across different domains.

\subsection*{Summary of Main Results} The main contribution of this work is to extend the properly discontinuous $\rho(\Gamma)$–action
from the finite-dimensional domain of discontinuity $\Omega_\rho$ to several natural
{infinite-dimensional} geometric spaces associated with it, including differential
forms with coefficients in pullback bundles, spaces of connections, pseudo-connections,
and, when $\Omega_\rho$ is a complex curve, Higgs bundle structures.

In this setting, we also introduce a new dynamical invariant of an Anosov representation:
a  free abelian group $F_{\rho}^{ab}$,  endowed with a rich topological structure,  and  generated by the holomorphic line bundles
$\mathcal E^g = g^{*}T^{1,0}\Omega_\rho$ and their duals.  

We further show that these constructions behave functorially with respect to morphisms
of Anosov representations, leading to a canonical functor from the category of
$(P^+,P^-)$-Anosov representations to the category of abelian groups.

\subsection*{Structure of the paper}We recall the necessary preliminaries in \S\ref{prel}. In \S\ref{sec2.1}, we construct the relevant vector bundles and study domains of discontinuity, along with connections on these bundles. \S\ref{Hl} focuses on the complex setting, where we build holomorphic vector bundles and show the properly discontinuous action on their pseudo-connections.  In \S\ref{cHg}, we establish the analogous result for Higgs bundle structures on holomorphic line bundles and prove \thmref{thm1.1}. In  \S\ref{cc}, we construct the free group $F$ and its abelianization $F^{ab}$, prove \thmref{thm5.2} and  develop the categorical framework yielding the functorial correspondence to free abelian groups and the proof of \thmref{thm1.3} in this section.
We also show that the isomorphism class of $F^{ab}$ is well defined over the character variety of Zariski dense Anosov representations in \S\ref{subsec2.5}.

\section{Preliminaries} \label{prel}

\subsection{Anosov Representations} Let $G$ be a semisimple Lie group and $(P^{+},P^{-})$ be a pair of opposite parabolic subgroups of $G$ (as defined in \cite[\S 3.2]{gw}). Set, $\mathcal F^+=G/P^+$, resp. $ \mathcal F^-= G/P^-$. The subgroup $L=P^+ \cap  P^- $ is the Levi subgroup of both $P^+$ and $P^-$. The homogeneous space $\chi= G/L$ can be embedded into $G/P^+\times G/P^-$  in   the following way: 

Define a group action of $G$ on $G/P^+\times G/P^-$ such that $g.(\bar{x},\bar{y})=(\bar{gx},\bar{gy}) \; \forall\, g \in G,\bar{x} \in G/P^+, \bar{y} \in G/P^- $. It is easy to check that $\st\,(id_{G/P^+},id_{G/P^-})=P^+ \cap  P^-=L$. Therefore by the orbit-stabilizer theorem $G/L=G\,(id_{G/P^+},id_{G/P^-}) $. That is,  the homogeneous space  $\chi= G/L \subset G/P^+\times G/P^-$. Note that $\chi$ is also the unique open $G$-orbit under this action, cf.  \cite[\S 2.1]{gw}.

Now take $\Gamma$ to be a finitely generated word hyperbolic group. Let $\partial_{\infty} \Gamma$ be the  boundary at $\infty$. Set, $\partial_{\infty} \Gamma^{(2)} =\partial_{\infty} \Gamma \times\partial_{\infty} \Gamma \smallsetminus \{{(t,t)~|~ t\in \partial_{\infty} \Gamma}\} $. Also note that a pair of points $(x_+,x_-)\in\mathcal{F}^+\times \mathcal{F}^-  $ is said to be \emph{transverse} if $(x_+,x_-)\in \chi \subset \mathcal{F}^+\times \mathcal{F}^-  $.
\begin{theorem} $($\cite[Theorem 8.3.C]{Gr}, \cite[Theorem 60]{Mi}, \cite[Theorem 2.9]{gw}$)$ Let $\Gamma$ be a finitely generated word hyperbolic group. Then there exists a proper hyperbolic metric space $\hat{\Gamma}$ such that 

$(1)$  $ \Gamma \times \mathbb{R}  \rtimes {\mathbb{Z} / 2\mathbb{Z}}$ acts on $\hat{\Gamma}$.

$(2)$ The $\Gamma \times {\mathbb{Z}/2\mathbb{Z}} $ action is isometric.

$(3)$ Every orbit $\Gamma \to \hat{\Gamma}$ is a quasi-isometry. In particular, $ \partial_{\infty} \hat{\Gamma} \cong \partial_{\infty} \Gamma $.

    $(4)$ The $\mathbb R$ -action is free and every orbit $\mathbb{R} \to \hat{\Gamma}$ is a quasi-isometric embedding. The induced map ${\hat{\Gamma } / \mathbb R}\to \partial_{\infty}{\hat{ \Gamma}}^{(2)} =\partial_{\infty} \hat{\Gamma} \times\partial_{\infty} \hat{\Gamma} \smallsetminus \{(\hat{t},\hat{t})~|~ \hat{t}\in \partial_{\infty} \hat{\Gamma}\} $ is a homeomorphism.
\end{theorem}
In fact $\hat{\Gamma}$ is unique up to a $ \Gamma \times {\mathbb{Z} / 2\mathbb{Z}}$ -equivariant quasi-isometry sending $\mathbb R$-orbits to $\mathbb R$-orbits. Denote by $\phi_t$ the $\mathbb R$-action on $\hat{\Gamma}$ and by $(\tau ^+,\tau^-):\hat{\Gamma} \to {\hat{\Gamma } / \mathbb R} \cong \partial_{\infty}{\hat{ \Gamma}}^{(2)} \cong \partial_{\infty} \Gamma^{(2)}$ the maps associating to a point the endpoints of its $\mathbb R$-orbit.
\begin{definition}\cite[Definition 2.10]{gw}  
A representation $\rho: \Gamma \to G $ is said to be  $(P^+,P^-)$-Anosov if there exists continuous $\rho$-equivariant maps $\xi^{+}, \hbox{resp. } \xi^{-}: \partial_{\infty} \Gamma \to \mathcal F^+, \hbox{ resp.  } \mathcal F^{-}$ such that:
\begin{enumerate}
\item  $\forall (t^+,t^-) \in \partial_{\infty} \Gamma^{(2)} ,\: (\xi^+(t^+),\xi^-(t^-))$ is transverse.

\item For one (and hence any) continuous and equivariant family of norms ${(\,{\|\,.\,\|}_{\hat{m}}\,)}_{{\hat{m}}\in \hat{\Gamma}}$ on $(\,T_{\xi^+({\tau ^+ }(\hat{m}))} \mathcal{F}^+)_{\hat{m} \in \hat{\Gamma}}$ (\,\hbox{resp.}\,$(\,T_ {\xi ^-(\tau ^-(\hat{m}))} \mathcal{F} ^-)_{\hat{m} \in \hat{\Gamma}}$\,), there exists $ A,a>0  $  such that $\forall t \geq 0,\,\hat{m}\in\hat{\Gamma}$ \,and \, $ e\in (\, T_{\xi^+(\tau^+(\hat{m}))} \mathcal{F}^+)_{\hat{m} \in \hat{\Gamma} }$ (\,\hbox{resp.} $e \in (\,T_ {\xi ^-{(\tau ^-(\hat{m}))}} \mathcal{F} ^-)_{\hat{m} \in \hat{\Gamma}}$\,):\\ $ {\|e\|}_{\phi_{-t} \hspace{0.06cm} (\hat{m})} \leq Ae^{-at} {\|e\|}_{\hat{m}}$ ( \hbox{resp.}$ {\|e\|}_{\phi _t \hspace{0.06cm} (\hat{m})} \leq Ae^{-at} {\|e\|}_{\hat{m}}$ ) .
 \end{enumerate}
 The maps $\xi^{\pm}$ are called its \emph{limit maps}.
\end{definition}
\begin{remark}
 A parabolic subgroup  $P=P^+$ is  self-opposite  means, $P$ is conjugate to its opposite parabolic  subgroup $P^-$. In that case, the two homogeneous spaces $\mathcal F^+=G/P^+$ and  $ \mathcal F^-= G/P^-$ are canonically identified. Then, there is a single \emph{limit map} $\xi^{+}= \xi^{-}=\xi: \partial_{\infty} \Gamma \longrightarrow \mathcal F^+= \mathcal F^{-}$, cf. \cite[\S 4.5]{gw}. 
\end{remark}

\subsection{Domains of discontinuity}\label{dod}
Suppose, $P$ is  a self-opposite parabolic subgroup of a semisimple Lie group $G$ and $\rho: \Gamma \to G $ is a $P$-Anosov representation with associated  limit map $\xi_\rho$. Also, assume that $P\backslash G/H$ is finite, for  a closed  subgroup $H$ of $G$. For example, this always happens, when $H$ is a parabolic subgroup of $G$ \cite[Theorem 7.40]{kn}, or when $H$ is symmetric \cite{wo}. 

In \cite[\S 3.1, 4.1]{cs}, the authors constructed domain of discontinuity $\Omega_{\rho}^I$ for $\rho(\Gamma)$ which extends  results in \cite{gw} or \cite{klp}. We outline the construction very briefly here. 

  \medskip  Denote, $G/H$ by $\mathcal{X}$. The $G$-orbit of a point $(m,x)$ in $G/P \times \mathcal{X}$ will be denoted by $\textbf{pos}(m,x)$. The map $\textbf{pos}(g_1 P,g_2 H) \longrightarrow P{g_1}^{-1}g_2 H$ gives an identification between the $G$-orbits of $(G/P \times \mathcal{X})$ and the elements of $P\backslash G/H$. The set $P\backslash G/H$ carries a partial order in the following way: for any two elements $PgH$ and $Pg'H$ of $P\backslash G/H$, we have $PgH \leq Pg'H$ if and only if $PgH \subset \overline{Pg'H}$. 
A subset $I\subset P\backslash G/H$ is called an ideal if for every $p\in I$ and every $p'\leq p$, we have $p' \in I$.

An open $\rho(\Gamma)$ invariant subset $\Omega \subseteq \mathcal{X}$ is called a domain of discontinuity for $\rho$ if for every compact subset $K \subset \Omega$, the cardinality of $\{\gamma\in\Gamma : {\rho(\gamma)K}\cap K \neq \emptyset \}$ is finite. 
In  \cite[Theorem 6.3]{cs}, it is shown that $$\Omega_\rho^I = \mathcal{X} \smallsetminus \bigcup_{z\in {\partial_{\infty} \Gamma}}\{x\in \mathcal{X}: \hbox{pos}(\xi_\rho(z),x)\in I)\}$$
is a domain of discontinuity for $\rho$ under certain conditions on the ideal $I$.

\subsection{Character Variety}
Suppose $G$ is a complex semisimple algebraic Lie group  and  $\Gamma$ is a finitely generated group. Define $Hom(\Gamma,G)$ to be the set of  homomorphisms from $\Gamma$ to $G$. The group $G$ acts on $Hom(\Gamma,G)$ by conjugation, and  the categorical (GIT) quotient ${\mathcal X}=Hom(\Gamma,G)//G$ is called  the  $G$-\emph{character variety} of $\Gamma$. In particular, ${\mathcal X}=Hom^{cr}(\Gamma,G)/G$, where $Hom^{cr}(\Gamma,G)$ denotes the set of completely reducible representations, cf. \cite[\S 11]{si}.

Let us denote,  $Hom_{P}^{zd} (\Gamma, G)$  as the set of Zariski dense $P$-Anosov representations. The character variety of the Zariski dense $P$-Anosov representations means the quotient space $Hom_{P}^{zd} (\Gamma, G)/G$. For more details, see \cite[\S 3.2]{gn}. 

In this paper, we assume familiarity with basic notions of vector bundles,  connections, pseudo-connections, holomorphic structures, Higgs bundles. The standard references for these topics are \cite{gh}, \cite{gu}, \cite{hu}, \cite{ra}.

\section{Proof of \thmref{2.4}}
Let, $\rho: \Gamma \to G $ be an Anosov representation. Then $\rho(\Gamma)$ acts properly discontinuously on an open subset (denoted by $\Omega_\rho$ ) of $\mathcal{X}=G/P$ ($P$ is a parabolic subgroup) or $G/H$ ($H$ is symmetric). These domains of discontinuity  have been well studied in \cite[\S 7-10]{gw} and \cite[\S 6]{cs}.

\subsection{Vector Bundles}\label{sec2.1}
Denote, $\Omega_\rho:=M$. Now, take the tangent bundle $E=TM$ over $M$. For each $g\in{\rho(\Gamma)}$, where $g=\rho(\gamma)$ for some $\gamma \in \Gamma$, construct the pullback bundle of $E$ by $\rho(\gamma)$. It will be denoted by 
$$E^{\rho(\gamma)}={\rho(\gamma)}^*E=\{(m,x)\in M\times TM: x \in T_{\rho(\gamma)m}M\},$$ where $\rho(\gamma): M \to M$ is a diffeomorphism for each $\gamma\in \Gamma$. 

Consider the family of vector bundles $\{E^{g}\}_{g\in \rho(\Gamma)}$ over $M$. We define a vector bundle isomorphism  
$$
\rho(\gamma) : E \longrightarrow E^{\rho(\gamma)}
$$  
by its action on fibers:  
$$
\begin{aligned}
\rho(\gamma)\Big|_m : E|_m = T_m M &\longrightarrow T_{\rho(\gamma)m} M = E^{\rho(\gamma)}|_m,\\
x &\longmapsto d\rho(\gamma)|_m \, x, \quad \text{for all } x \in T_m M.
\end{aligned}
$$

Similarly, for a composition of elements $\gamma, \gamma' \in \Gamma$, we define  
$$
\rho(\gamma) : E^{\rho(\gamma')} \longrightarrow E^{\rho(\gamma\gamma')}
$$  
by its fiberwise action:  
$$
\begin{aligned}
\rho(\gamma)\Big|_m : T_{\rho(\gamma')m} M &\longrightarrow T_{\rho(\gamma\gamma') m} M,\\
x &\longmapsto d\rho(\gamma)|_{\rho(\gamma') m} \, x, \quad \text{for all } x \in T_{\rho(\gamma')m} M.
\end{aligned}
$$

Thus, $\rho(\Gamma)$ acts on the family $\{E^g\}_{g \in \rho(\Gamma)}$ via vector bundle isomorphisms.

We can extend this action to the endomorphism bundles $\{ \operatorname{End}(E^g) \}_{g \in \rho(\Gamma)}$ as follows:  

Define  
$$
\rho(\gamma) : \operatorname{End}(E^{\rho(\gamma')}) \longrightarrow \operatorname{End}(E^{\rho(\gamma\gamma')})
$$  
by its fiberwise action:  
$$
\begin{aligned} 
\rho(\gamma)\Big|_m : \operatorname{End}(T_{\rho(\gamma') m} M) &\longrightarrow \operatorname{End}(T_{\rho(\gamma\gamma') m} M),\\
v &\longmapsto w, \quad \end{aligned}
$$
where $$ w(x) = d\rho(\gamma)|_{\rho(\gamma') m} \circ v \circ \big(d\rho(\gamma)|_{\rho(\gamma') m}\big)^{-1} (x), \quad \forall x \in T_{\rho(\gamma\gamma') m} M.
$$

Let us denote the space of $k$-forms with coefficients in $End(E^{\rho(\gamma)})$ by 
$$
\Omega^k(M; End(E^{\rho(\gamma)})), \quad \forall \, k \in \{0,1,\dots, \dim M\}.
$$

Observe that 
$$
\rho(\gamma) : \Omega^k(M; End(E^{\rho(\gamma')})) \longrightarrow \Omega^k(M; End(E^{\rho(\gamma\gamma')}))
$$
is defined on decomposable elements by
$$
\rho(\gamma)\Big(\sum_i \alpha_i \sigma_i\Big) = \sum_i \alpha_i \, \rho(\gamma)(\sigma_i),
$$
where $\alpha_i \in \Omega^k(M)$ and $\sigma_i \in \Omega^0(M; End(E^{\rho(\gamma')}))$.  

From this, it follows that $\rho(\Gamma)$ acts on 
$$
\coprod_{g \in \rho(\Gamma)} \Omega^k(M; End(E^g)).
$$

\medskip Note that the spaces of $k$-forms, connections, and Higgs bundles we consider are infinite-dimensional, and proper discontinuity does not automatically extend to such settings. The following lemma allows us to reduce this to the finite-dimensional base case via a carefully chosen equivariant projection. This mechanism underlies the proofs of the theorems. Even though the lemma is easy to prove,  we include the proof for completeness.  
\begin{lemma}\label{ml}
Suppose, $\psi:X\longrightarrow Y$ is a continuous map between two topological spaces on which  $\rho(\Gamma)$ acts. Also, assume that $\psi$ is $\rho(\Gamma)$-equivariant. Then, if    $\rho(\Gamma)$ acts properly discontinuously on $Y$, then it also acts on $X$ properly discontinuously.
\end{lemma}
\begin{proof}
We claim that for any compact set $K \subset X$, the set 
$$
\{\gamma \in \Gamma : \rho(\gamma)K \cap K \neq \emptyset\}
$$
is finite.  

Suppose not; then for some compact $K \subset X$, this set is infinite. Since $\psi$ is continuous, $\psi(K)$ is compact in $Y$, and for any $\gamma \in \Gamma$,  
$$
\psi(\rho(\gamma)K \cap K) \subseteq \psi(\rho(\gamma)K) \cap \psi(K) = \rho(\gamma)\psi(K) \cap \psi(K).
$$
Hence, $\rho(\gamma)K \cap K \neq \emptyset$ implies $\rho(\gamma)\psi(K) \cap \psi(K) \neq \emptyset$.  

If the original set were infinite, then 
$$
\{\gamma \in \Gamma : \rho(\gamma)\psi(K) \cap \psi(K) \neq \emptyset\}
$$ 
would also be infinite, contradicting the proper discontinuity of the $\rho(\Gamma)$-action on $Y$.  

Therefore, $\rho(\Gamma)$ acts properly discontinuously on $X$.
\end{proof}

\medskip Now, we will prove the following.
\begin{proposition}\label{2.3}
   For each  $k\in \{ 0,1,\dotso, dim\;M \}$, $\rho(\Gamma)$ acts properly discontinuously on the disjoint union $\coprod_{g\in \rho(\Gamma)}{\Omega^k(M;End(E^g))}$. 
\end{proposition}

\begin{proof}
For each $g \in \rho(\Gamma)$, define a topology on the space $\Omega^k(M; \mathrm{End}(E^g))$ using pointwise convergence. That is, for any section $\sigma \in \Omega^k(M; \mathrm{End}(E^g))$, we view $\sigma$ as a function
$$
\sigma: M \to \bigcup_{m \in M} {\left\{\bigwedge^k T^* M \otimes \mathrm{End}(E^g)\right\}}_m$$$$ \hbox{where }\quad {\left\{\bigwedge^k T^* M \otimes \mathrm{End}(E^g)\right\}}_m := \bigwedge^k T^* M \otimes \mathrm{End}(E^g)\,\, \forall\, m\in M, 
$$
and identify it with an element of the product space
$$
\prod_{m \in M} \left(\bigwedge^k T^* M \otimes \mathrm{End}(E^g)\right).
$$

Now define
$$
V := \coprod_{g \in \rho(\Gamma)} \prod_{m \in M} \left( \bigwedge^k T^* M \otimes \mathrm{End}(E^g) \right),
$$
and let
$$
W := \coprod_{g \in \rho(\Gamma)} \Omega^k(M; \mathrm{End}(E^g)) \subset V
$$
inherit the subspace topology from $V$. This topology is the topology of pointwise convergence on sections.

For each $g \in \rho(\Gamma)$, define the map
$$
\pi^g : \prod_{m \in M} \left( \bigwedge^k T^* M \otimes \mathrm{End}(E^g) \right) \to M
$$
as the composition
$$
\pi^g = g \circ p^g \circ \pi_{m'},
$$
where: \( m' \in M \) is fixed, \( \pi_{m'} \) is the projection onto the $m'$-th coordinate,  \( p^g \) is the bundle projection from the $m'$-th coordinate  \( \bigwedge^k T^* M \otimes \mathrm{End}(E^g) \) to \( M \),
and \( g \) acts on \( M \) via the representation \( \rho(\Gamma) \).

This map \( \pi^g \) is continuous.

We now define a global map
$$
\pi : V \to M, \quad \pi(f) := \pi^g(f) \quad \text{if } f \in \prod_{m \in M} \left( \bigwedge^k T^* M \otimes \mathrm{End}(E^g) \right).
$$
This is well-defined and continuous on \( V \), and hence on its subset \( W \).
Next, define an action of $\rho(\Gamma)$ on $V$ as follows. For $g' \in \rho(\Gamma)$ and 
$$
f \in \prod_{m \in M} \Bigl( \bigwedge^k T^*_m M \otimes \mathrm{End}(E^g) \Bigr),
$$ 
set
$$
(g' \cdot f)(m) := (\mathrm{id} \otimes g') \big(f(m)\big),
$$ 
where $\mathrm{id}$ acts on $\bigwedge^k T^*_m M$ and 
$$
g' : \mathrm{End}(E^g) \to \mathrm{End}(E^{g'g})
$$ 
is induced by pullback along the group action.

This defines a continuous action on $V$ that restricts to $W \subset V$.

Now consider \( \pi : W \to M \) as above. For any \( \sigma \in \Omega^k(M; \mathrm{End}(E^g)) \subset W \), we have
$$
\pi(\sigma) = \pi^g(\sigma) = g(m'),
$$
since \( p^g \circ \sigma = \mathrm{id}_M \). The diagram
$$
\begin{tikzcd}[row sep=huge]
W \supset \Omega^k(M; \mathrm{End}(E^g)) \ni \sigma \arrow[r, "g'"] \arrow[d, "\pi" ] & g' \cdot \sigma \in \Omega^k(M; \mathrm{End}(E^{g'g})) \subset W \arrow[d, "\pi" ] \\
M \ni g(m') \arrow[r, "g'"] & g' g(m') \in M
\end{tikzcd}
$$
commutes for all \( g, g' \in \rho(\Gamma) \), showing that \( \pi \) is \( \rho(\Gamma) \)-equivariant.

Finally, by applying Lemma~\ref{ml} with \( X = W \), \( Y = M \), and \( \psi = \pi \), and using the fact that \( \rho(\Gamma) \) acts properly discontinuously on \( M \), we conclude that the induced action on \( W \) is also properly discontinuous.

This completes the proof.
\end{proof}

  \subsection{Proof of \thmref{2.4}  }
  Let us denote the space of all connections on $E^g$ by $\mathcal{D}^{E^g}$, i.e.,
$$
\mathcal{D}^{E^g} := \{\text{all connections on } E^g\}.
$$
Note that, if $D: \Omega^\bullet(M;E^{\rho(\gamma')})\longrightarrow\Omega^\bullet(M;E^{\rho(\gamma')})$ is a connection of $E^{\rho(\gamma')}$, then $${\rho(\gamma)D}:\Omega^\bullet(M;E^{\rho(\gamma\gamma')})\longrightarrow\Omega^\bullet(M;E^{\rho(\gamma\gamma')})$$ is a connection of $E^{\rho(\gamma\gamma')}$ given by $$(\rho(\gamma)D)(\sum{\alpha_i}{\sigma_i})={\rho(\gamma)D(\sum{\alpha_i}\rho(\gamma^{-1}){\sigma_i})} \, \forall\,{\sum{\alpha_i}{\sigma_i}}\in \Omega^\bullet(M;E^{\rho(\gamma\gamma')}),$$ where $\alpha_i\in \Omega^\bullet(M), \sigma_i\in\Omega^0(M;E^{\rho(\gamma\gamma')})$. 

\medskip   In this way, $\rho(\Gamma)$ acts on $\coprod_{g\in{\rho(\Gamma)}} {\mathcal{D}^{E^g}}$. Now. we will show that the action is properly discontinuous.

\medskip Since $\mathcal{D}^{E^g}$ is an affine space over $\Omega^1(M; \mathrm{End}(E^g))$, its elements are in one-to-one correspondence with those of $\Omega^1(M; \mathrm{End}(E^g))$ for each $g \in \rho(\Gamma)$.  
Consequently, for a fixed connection $D \in \mathcal{D}^E$, we have a bijection
$$
\Psi : \Omega^1(M; \mathrm{End}(E)) \longrightarrow \mathcal{D}^E, \quad \Psi(A) = D + A \quad \forall A \in \Omega^1(M; \mathrm{End}(E)).
$$
Similarly,  define $$\Psi^g:\Omega^1(M;End(E^g))\longrightarrow \mathcal{D}^{E^g}$$ by $$\Psi^g(A)=gD+A\,\,\forall\,A\in{\Omega^1(M;End(E^g))}$$  where $gD\in {\mathcal{D}^{E^g}}$ for each $g\in \rho(\Gamma)$.

Give topological structures on ${\{\mathcal{D}^{E^g}\}}_{g\in\rho(\Gamma)}$ through these bijections ${\{\Psi^g\}}_{g\in\rho(\Gamma)}$, respectively. In the induced topologies on ${\{\mathcal{D}^{E^g}\}}_{g\in\rho(\Gamma)}$, ${\{\Psi^g\}}_{g\in\rho(\Gamma)}$ are homeomorphisms.

Define, $$\Phi:{ \coprod_{g\in{\rho(\Gamma)}}{\mathcal{D}^{E^g}}}\longrightarrow {\coprod_{g\in{\rho(\Gamma)}}{\Omega^1(M;End(E^g))}} $$ by $$\Phi\Big|_{\mathcal{D}^{E^g}}= {(\Psi^g)}^{-1}.$$ Clearly, $\Phi$ is continuous. For $C\in\mathcal{D}^{E^g}$, we have  $g'C\in \mathcal{D}^{E^{g'g}}$.  Now  observe that,
$$\Phi(g'C)=\Phi^{g'g}(g'C)= g'C-g'gD
=g'(C-gD)
={g'{\Phi^{g}}}(C)
=g'{\Phi(C)}.$$
Therefore, $\Phi$ is $\rho(\Gamma)$-equivariant.

From \propref{2.3}, we know  that $\rho(\Gamma)$ acts properly discontinuously on $$\coprod_{g\in{\rho(\Gamma)}}{\Omega^1(M;End(E^g))}.$$ Now, using \lemref{ml}, we can conclude our result. \qed

\medskip  We can also give the following corollary.
\begin{corollary}\label{cor1}
    $\rho(\Gamma)$ acts properly discontinuously on $\coprod_{g\in{\rho(\Gamma)}}{\mathcal{D}^{E^g}\big/ {\mathcal{A}ut(E^g)}}$.
\end{corollary}
\begin{proof}
Suppose $C, \widetilde{C} \in \mathcal{D}^{E^{g'}}$ with $C = \tau \cdot \widetilde{C}$ for some $\tau \in \mathcal{A}ut(E^{g'})$. Then, for any $g \in \rho(\Gamma)$, we have  
$$
gC = g \tau g^{-1} \cdot g \widetilde{C},
$$
where $gC, g\widetilde{C} \in \mathcal{D}^{E^{gg'}}$ and $g \tau g^{-1} \in \mathcal{A}ut(E^{gg'})$. Thus, the action of $\rho(\Gamma)$ is well-defined on  
$$
\coprod_{g \in \rho(\Gamma)} \mathcal{D}^{E^g} \big/ \mathcal{A}ut(E^g).
$$

Next, consider the map 
$$
\pi \circ \Phi : \coprod_{g \in \rho(\Gamma)} \mathcal{D}^{E^g} \longrightarrow M,
$$
where $\pi$ and $\Phi$ are defined in the proofs of \propref{2.3} and \thmref{2.4}. Note that  
$$
\pi \circ \Phi \Big|_{\mathcal{D}^{E^g}} = g m' \quad \text{for all } g \in \rho(\Gamma),
$$
where $m' \in M$ is fixed. But, $ \Phi$ is not well-defined on 
 $\mathcal{D}^{E^g} \big/ \mathcal{A}ut(E^g)$, for each $g \in \rho(\Gamma)$. Let us define a map $$\zeta: \coprod_{g \in \rho(\Gamma)} \mathcal{D}^{E^g} \longrightarrow M, $$ such that $\zeta\Big|_{\mathcal{D}^{E^g}} = g m'$. Hence, $\zeta$ is well-defined on $\coprod_{g \in \rho(\Gamma)} \mathcal{D}^{E^g} \big/ \mathcal{A}ut(E^g)$. It is also continuous and $\rho(\Gamma)$-equivariant.
 The conclusion follows immediately.
\end{proof}

\section{Holomorphic Vector bundles}\label{Hl}
Now, consider the case,
when $\Omega_\rho=M$ is a complex manifold.  
Let, $\mathcal{E}= T^{1,0}M$ be the holomorphic vector bundle over  $M$. Here, $T^{1,0}M$ denotes the $(1,0)$ part of the complexified tangent bundle ${TM}\otimes_\R \C=E\otimes_\R \C$.

Consider the vector bundle isomorphism $$g\otimes id:E^{g'}\otimes \C \longrightarrow E^{gg'}\otimes \C.$$ Note that, $$(dg\Big|_m \otimes {id\Big|_m})(T^{1,0} M\Big|_{g'm})=T^{1,0} M\Big|_{gg'm}\,\,\forall\,m\in\M,$$ (see \cite[Proposition 1.3.2]{hu}). Therefore, we can  construct the pullback bundle of $\mathcal{E}$ by $g$, that will be denoted by $$\mathcal{E}^{g}={g}^*\mathcal{E}=\{(m,x)\in M\times T^{1,0}M: x \in T^{1,0}M\Big|_{gm}\}.$$

Now, consider the family of holomorphic vector bundles $\{\mathcal{E}^{g}\}_{g\in {\rho(\gamma)}}$ over $M$. From  the above paragraph, it follows that $\rho(\Gamma)$  acts on $\{\mathcal{E}^g\}_{g\in{\rho(\Gamma)}}$ as holomorphic vector bundle isomorphisms, where $$g\mathcal{E}^{g'}=(g\otimes id)(\mathcal{E}^{g'})=\mathcal{E}^{gg'}\,\,\forall\,g,g'\in \rho(\Gamma).$$
Simultaneously, we can extend this action on the spaces of $(p,q)$-forms with coefficients in their endomorphism bundles, pseudo-connections, respectively.

Using similar ideas as in the proof of \propref{2.3} and \thmref{2.4},  the following is obtained.
 \subsection{Proof of \thmref{thm 3.4}  }
  Let us denote, the topological space $$\coprod_{g\in{\rho(\Gamma)}}{\prod_{m\in M}{(\bigwedge^p{T^*}^{1,0}M}\otimes\bigwedge^q{T^*}^{0,1}M\otimes{End(\mathcal{E}^g)})} \,\,\hbox{by}\,\: \mathcal{V}.$$ 
     Now, identify $\coprod_{g\in{\rho(\Gamma)}}{\Omega^{p,q}(M;End(\mathcal{E}^g))}$ as a subset of $\mathcal{V}$ for giving the  topology of pointwise convergence on it. Denote, this subset by $\mathcal{W}$.

For each $g\in{\rho(\Gamma)}$, define a map $$\pi^g:{\prod_{m\in M}{(\bigwedge^p{T^*}^{1,0}M}\otimes\bigwedge^q{T^*}^{0,1}M\otimes{End(\mathcal{E}^g)})}\longrightarrow {M}$$ such that $$\pi^{g}(f)={g\circ {p^g}\circ {\pi_{m'}}(f)}, \forall\, f \in {\prod_{m\in M}{(\bigwedge^p{T^*}^{1,0}M}\otimes\bigwedge^q{T^*}^{0,1}M\otimes{End(\mathcal{E}^g)})}.$$
Here,  $$\pi_{m'}:\prod_{m\in M}{(\bigwedge^p{T^*}^{1,0}M}\otimes\bigwedge^q{T^*}^{0,1}M\otimes{End(\mathcal{E}^g)})\longrightarrow {(\bigwedge^p{T^*}^{1,0}M}\otimes\bigwedge^q{T^*}^{0,1}M\otimes{End(\mathcal{E}^g)})$$ is the $m'$-th projection for a fixed $m'\in M$,\,\,\,   ${p^g}:{(\bigwedge^p{T^*}^{1,0}M}\otimes\bigwedge^q{T^*}^{0,1}M\otimes{End(\mathcal{E}^g)})\longrightarrow {M}$ is the vector bundle map onto $M$ and composition by $g$ is induced by  the action of $\rho(\Gamma)$ on $M$. Clearly, $\pi^g$ is a continuous map.

For any $f\in {{\prod_{m\in M}{(\bigwedge^p{T^*}^{1,0}M}\otimes\bigwedge^q{T^*}^{0,1}M\otimes{End(\mathcal{E}^g)})}}$ and $g'\in{\rho(\Gamma)}$, define $$g'f\in {{\prod_{m\in M}{(\bigwedge^p{T^*}^{1,0}M}\otimes\bigwedge^q{T^*}^{0,1}M\otimes{End(\mathcal{E}^{g'g})})}}$$
such that $g'(f(m))={(id\otimes {{g'}})(f(m))}$ at each $m\in M$. Here, $$ id:{\bigwedge^p{T^*}^{1,0}M}\otimes\bigwedge^q{T^*}^{0,1}M\longrightarrow {\bigwedge^p{T^*}^{1,0}M}\otimes\bigwedge^q{T^*}^{0,1}M$$ and $g': End(\mathcal{E}^g)\longrightarrow End(\mathcal{E}^{g'g})$. In this way, we can induce the action of $\rho(\Gamma)$ on $\mathcal{V}$. In particular,  $\rho(\Gamma)$ acts on $\mathcal{W}$. 

 Let, $\pi:\mathcal{W}\longrightarrow M$ be such that $\pi(f)=\pi^g(f)$, when $$f\in {{\prod_{m\in M}{(\bigwedge^p{T^*}^{1,0}M}\otimes\bigwedge^q{T^*}^{0,1}M\otimes{End(\mathcal{E}^g)})}}.$$ It is easy to check that, $\pi$ is a $\rho(\Gamma)$-equivariant continuous map. Therefore, it follows from  \lemref{ml} that, the action is properly discontinuous on $$\coprod_{g\in{\rho(\Gamma)}}{\Omega^{p,q}(M;End(\mathcal{E}^g))}.$$

Let $\mathfrak{D}^{\mathcal{E}^g}$ denote the space of pseudo-connections on $\mathcal{E}^g$. It is an affine space over  $\Omega^{0,1}(M; \mathrm{End}(\mathcal{E}^g))$. We have just shown that $\rho(\Gamma)$ acts properly discontinuously on 
$$
\coprod_{g \in \rho(\Gamma)} \Omega^{0,1}(M; \mathrm{End}(\mathcal{E}^g)).
$$
The remainder of the proof then proceeds analogously to that of Theorem~\ref{2.4}. \qed
    \section{ Domain of Discontinuity over a complex curve}\label{cc}

    In this section, we will focus on the case when $M=\Omega_{\rho}$ is a complex curve, i.e. ${dim}_{\C}\;-M=1$. For example, when $G/P = Gr_1(\C^2)$ is the Grassmannian of $1$-planes in $\C^2$, then   
 ${dim}_{\C}\;\text{Gr}_1(\C^2)=1$. This implies ${dim}_{\C}\;M=1$.
\subsection{Construction of Higgs Bundles}\label{cHg}
 A Higgs bundle structure on a vector bundle $\mathbb{E}$ over a complex manifold $M$  is a pair $(D, \phi)$ where $D$ is a pseudo-connection on $\mathbb{E}$ and $\phi$ is a (1,0)-form with coefficients in $End(\mathbb{E})$ ($\phi\in \Omega^{1,0}(M;End(\mathbb{E}))$ such that $(D)^2=0$, $D^{End(\mathbb{E})}\phi=0$ and $\phi\wedge\phi=0$. 

 Since, in our case, ${dim}_{\C}\;M=1$, the conditions $(D)^2=0$ and $\phi\wedge\phi=0$ are satisfied for any pseudo-connection $D$ on $\mathbb{E}$ and  $\phi\in \Omega^{1,0}(M;End(\mathbb{E}))$. In particular, the holomorphic vector bundles $\{\mathcal{E}^{g}\}_{g\in {\rho(\gamma)}}$ (defined in \S\ref{Hl}) are naturally equipped with Higgs bundle structures. Now, we will prove \thmref{thm1.1}.

 \subsubsection{Proof of \thmref{thm1.1}  }
  Let us denote, the set of all possible Higgs bundle structures on  $\{\mathcal{E}^g\}_{g\in{\rho(\Gamma)}}$  by $\mathcal{H}$. Clearly, $$\mathcal{H}=\coprod_{g\in{\rho(\Gamma)}}\big\{(D,\phi)\in{\mathfrak{D}^{\mathcal{E}^g}}\times {\Omega^{1,0}(M;End(\mathcal{E}^g))} : D^{End(\mathcal{E}^g)}\phi=0\big\}.$$

 Note that, if $(D,\phi)\in{\mathfrak{D}^{\mathcal{E}^{g'}}}\times {\Omega^{1,0}(M;End(\mathcal{E}^{g'}))}$ such that $D^{End(\mathcal{E}^{g'})}\phi=0$, then $$(gD, g\phi)\in{\mathfrak{D}^{\mathcal{E}^{gg'}}}\times {\Omega^{1,0}(M;End(\mathcal{E}^{gg'}))}$$ satisfies ${gD}^{End(\mathcal{E}^{gg'})}g\phi=0$. Therefore,  $\rho(\Gamma)$ acts diagonally on $\mathcal{H}$. 
 We can induce the subspace topology on $\mathcal{H}$ from $$\coprod_{g\in{\rho(\Gamma)}}({\mathfrak{D}^{\mathcal{E}^g}}\times {\Omega^{1,0}(M;End(\mathcal{E}^g))}.$$  
 
 From \thmref{thm 3.4},  we can conclude that the action of $\rho(\Gamma)$ on $$\coprod_{g\in{\rho(\Gamma)}}({\mathfrak{D}^{\mathcal{E}^g}}\times {\Omega^{1,0}(M;End(\mathcal{E}^g))}$$ is properly discontinuous.
Now, take the embedding $$i: \mathcal{H}\longrightarrow\coprod_{g\in{\rho(\Gamma)}}({\mathfrak{D}^{\mathcal{E}^g}}\times {\Omega^{1,0}(M;End(\mathcal{E}^g))}$$ and use \lemref{ml} to get into our conclusion. \qed

We can also give the following corollary.
\begin{corollary}
    The action of $\Gamma$ by $\rho$ is  properly discontinuous on $$\coprod_{g\in{\rho(\Gamma)}}{\{\hbox{Higgs bundle structures on}\;\mathcal{E}^g\}\big/ {\mathcal{A}ut(\mathcal{E}^g)}}.$$
\end{corollary}
\begin{proof}
    Using similar kinds of argument as in the proof of \corref{cor1}, we can prove this.
\end{proof}
 
 \subsection{Construction of a Free Group}\label{5.2}

When $M := \Omega_\rho $ is a complex manifold of complex dimension $1$, the pullback bundles $\{\mathcal{E}^g\}_{g \in \rho(\Gamma)}$ naturally define holomorphic line bundles on $M$.

We define an equivalence relation on $\rho(\Gamma)$ by setting $g \sim g'$ if and only if $g(m) = g'(m)$ for all $m \in M$. Let $\overline{\rho(\Gamma)} := \rho(\Gamma)/{\sim}$ denote the resulting quotient. Since $\mathcal{E}^g = \mathcal{E}^{g'}$ whenever $g \sim g'$, each equivalence class corresponds uniquely to a holomorphic line bundle on $M$.

We now construct a group $F$ generated by these line bundles and their duals, modeled after the Picard group of a complex curve. Specifically, consider the set generated by formal tensor products of elements from
$$
\left\{ \mathcal{E}^g, (\mathcal{E}^g)^*, \mathcal{O} \right\}_{[g] \in \overline{\rho(\Gamma)}},
$$
where $\mathcal{O}$ denotes the trivial line bundle and $(\mathcal{E}^g)^*$ the dual of $\mathcal{E}^g$.

We define an equivalence relation on this set by identifying expressions that differ by the following canonical isomorphisms of line bundles:
$$
L \otimes L^* \cong \mathcal{O}, \quad L \otimes \mathcal{O} \cong L, \quad \mathcal{O} \otimes \mathcal{O} \cong \mathcal{O},
$$
for any holomorphic line bundle $L$. The resulting equivalence classes form a group $F$, with the group operation given by tensor product, identity element $\mathcal{O}$, and inverses given by dualization. This group is free on the generators $\{\mathcal{E}^g\}_{[g] \in \overline{\rho(\Gamma)}} \cup \{(\mathcal{E}^g)^*\}_{[g] \in \overline{\rho(\Gamma)}}$, subject only to the identifications above.

We refer to expressions in $F$ that contain no adjacent canceling pairs $\mathcal{E}^g \otimes (\mathcal{E}^g)^*$ or $(\mathcal{E}^g)^* \otimes \mathcal{E}^g$ as \emph{reduced words}. For example, a non-identity element of $F$ will be of the form $$\mathcal{E}^{g_1} \otimes \dotso\mathcal{E}^{g_i} \otimes {\mathcal{E}^{g_j}}^* \otimes \mathcal{E}^{g_k}\otimes\dotso\otimes\mathcal{E}^{g_n} $$ where $\mathcal{E}^{g_i}\neq\mathcal{E}^{g_j} $ and $\mathcal{E}^{g_j}\neq \mathcal{E}^{g_k} $.  From now on, we identify elements of $F$ with such reduced tensor expressions.

Let $\widetilde{M} := M \cup \{\mathcal{O}\}$. We equip $\widetilde{M}$ with a topology in which $\mathcal{O}$ is declared as an open singleton set, and the topology on $M \subset \widetilde{M}$ is standard. Our goal is to construct a topology on $F$ using this structure and to show that  $\rho(\Gamma)$ acts properly discontinuously on $F \setminus \{\mathcal{O}\}$.

\begin{theorem}\label{thm1.2}

    $\rho(\Gamma)$ acts as a group of automorphisms on the free group $F$.  Moreover, $F$ has a structure of non-trivial topological space so that  $\rho(\Gamma)$ acts on  $F$ in the following way: for any compact set $K \subset F$, the set
    $$
    \{\gamma \in \Gamma : \rho(\gamma)K \cap K \setminus \{\mathcal{O}\} \neq \emptyset \}
    $$
    is finite. In other words, $\rho(\Gamma)$ acts properly discontinuously on $F \setminus \{\mathcal{O}\}$.
\end{theorem}
\begin{proof}
We first define a topology on $F$. Fix a base point $m \in M$, and define a map $\pi: F \to \widetilde{M}$
as follows:
$$
\pi(\mathcal{O}) := \mathcal{O}, \quad
\pi(\mathcal{E}^{g_1} \otimes \cdots \otimes \mathcal{E}^{g_i} \otimes (\mathcal{E}^{g_j})^* \otimes \cdots \otimes \mathcal{E}^{g_n}) := g_1(m),
$$
where $g_1$ is the first index appearing in the reduced word. This is well defined because if $\mathcal{E}^{g_1} = \mathcal{E}^{g_1'}$, then $g_1(m) = g_1'(m)$ by definition of the equivalence relation.

We now define the topology on $F$ to be the \emph{initial topology induced by $\pi$}: that is, a subset $U \subseteq  F$ is open if and only if there exists an open set $V \subseteq   \widetilde{M}$ such that $U={\pi}^{-1}(V)$. In particular, the preimage under $\pi$ of any open subset of $M$ or the point $\mathcal{O}$ defines an open subset of $F$.
\subsubsection*{Proper Discontinuity of the Action of \texorpdfstring{$\rho(\Gamma)$}{ρ(Γ)} on \texorpdfstring{$F$}{F}.}

The group $\rho(\Gamma)$ acts on $F$ by:
$$
g\cdot(\mathcal{O}):=\mathcal{O}, \quad g \cdot (\mathcal{E}^{g_1} \otimes \cdots \otimes (\mathcal{E}^{g_j})^* \otimes \cdots) := \mathcal{E}^{g g_1} \otimes \cdots \otimes (\mathcal{E}^{g g_j})^* \otimes \cdots.
$$
This action preserves equivalence classes and respects the group structure. It also lifts the natural action of $\rho(\Gamma)$ on $M$, extended trivially to $\mathcal{O} \in \widetilde{M}$.

Moreover, the projection $\pi: F \to \widetilde{M}$ is $\rho(\Gamma)$-equivariant:
$$
\pi(g \cdot x) = g \cdot \pi(x), \quad \forall g \in \rho(\Gamma),\, x \in F.
$$

We know that, the action of $\rho(\Gamma)$ on $M$ is properly discontinuous,  Applying \lemref{ml} (the equivariant proper discontinuity lemma), we conclude that the action of $\rho(\Gamma)$ on $F \setminus \{\mathcal{O}\}$ is properly discontinuous.
\end{proof}

\subsection{Abelianization of the Free Group}\label{ab}

Let us denote the abelian group corresponding to the free group $F$ by $F^{ab}$. Specifically, $F^{ab}$ is freely generated by the line bundles \[
\left\{ \mathcal{E}^g, (\mathcal{E}^g)^* \right\}_{[g] \in \overline{\rho(\Gamma)}},
\]
 with the trivial line bundle $\mathcal{O}$ as the identity element. We refer to expressions in $F^{ab}$ that contain no  canceling pairs as reduced words. For example, a non-identity element of $F^{ab}$ will be of the form $L_1 \otimes \dotso L_i \otimes  \dotso L_k\otimes\dotso\otimes L_n $ where $L_i \neq L_k^{-1}$  and $L_i, L_k \in \big\{ \mathcal{E}^g, (\mathcal{E}^g)^* \big\}_{[g] \in \overline{\rho(\Gamma)}}$ for any $i, k \in \{1,2, \dotso n\}$. From now on, we identify elements of $F^{ab}$ with such reduced tensor expressions.

Let $M^{(n)} := M \times \dotso \times M$ (direct product of $n$-copies of $M$), equipped with product topology. Define an equivalence relation on $M^{(n)}$ in the following way: 
$$ (m_1, m_2, \dotso m_n) \sim (m_{\sigma(1)},  m_{\sigma(2)}, \dotso m_{\sigma(n)})\,\, \hbox{for all}\,\, \sigma \in S_n.$$
Let,  $\overline{M}^{(n)} := M^{(n)}/{\sim}$ denote the resulting quotient, and $q_n:M^{(n) } \to \overline{M}^{(n)} $ be the quotient map inducing the quotient  topology on $\overline{M}^{(n)}$. Also, denote $$\mathcal{M}:= \coprod_{n \in \mathbb{N}} M^{(n)} \,\,\hbox{and}, \,\,\overline{\mathcal{M}}:= \coprod_{n \in \mathbb{N}} \overline{M}^{(n)}.$$ The two sets are each endowed with the disjoint union topology. 
We note the following  lemmas which are easy to see.
 \begin{lemma}\label{ml2}
     Let, $\big\{\mathcal{A}_i\big\}_{i \in \mathcal{I}}$, $\big\{\mathcal{B}_i\big\}_{i \in \mathcal{I}}$ be a collection of topological spaces, where $\mathcal{I}$ is an index set (finite or infinite). Moreover, $f_i: \mathcal{A}_i \longrightarrow \mathcal{B}_i$ is  a proper map for each $i \in \mathcal{I}$. Suppose, $f: \coprod_{i \in \mathcal{I}} \mathcal{A}_i \longrightarrow  \coprod_{i \in \mathcal{I}} \mathcal{B}_i$ is defined by $f\Big|_{\mathcal{A}_i} := f_i$. Then, $f$ is also proper.
 \end{lemma}
 \begin{lemma} \label{ml3}
      Suppose, $\phi:X\longrightarrow Y$ is a proper surjective map between two topological spaces on which  $\rho(\Gamma)$ acts. Also, assume that $\phi$ is  $\rho(\Gamma)$-equivariant. Then, if    $\rho(\Gamma)$ acts properly discontinuously on $X$, then it also acts on $Y$ properly discontinuously. 
  \end{lemma}

\subsection{Proof of \thmref{thm5.2}}\label{pf}
    Let us denote, $\overline{F} := F^{ab}\setminus \{\mathcal{O}\}$. We first define a topology on $\overline{F}$. Fix a base point $m \in M$, and define a map: 
    \[\pi: \overline{F} \to \overline{\mathcal{M}}
\]
as follows:
\[\pi (\mathcal{E}^g):= gm, \quad \pi (L_1 \otimes \dotso L_i \otimes \dotso\otimes L_n) := (\pi(L_1), \dotso \pi(L_i),\dotso \pi(L_n)),\]
where $L_i$ is the $i$-th index appearing in the reduced word and each  $L_i \in \big\{ \mathcal{E}^g, (\mathcal{E}^g)^* \big\}_{[g] \in \overline{\rho(\Gamma)}}$ for all $i \in \{1,2, \dotso n\}$.
For example, $$\pi(\mathcal{E}^{g} \otimes \cdots (\mathcal{E}^{h})^* \otimes \cdots \otimes \mathcal{E}^{k}) := (g (m), \dotso h (m),\dotso k (m)).$$
 
 Note that, $$L_1 \otimes \dotso L_i \otimes \dotso\otimes L_n= L_{\sigma(1)}\otimes \dotso L_{\sigma(i)} \otimes \dotso\otimes L_{\sigma(n)}, \,\, \forall \, \sigma \in S_n. $$

  In particular,  $\pi(L_{\sigma(1)}\otimes \dotso L_{\sigma(i)} \otimes \dotso\otimes L_{\sigma(n)})=(\pi(L_{\sigma(1)}, \dotso \pi(L_{\sigma(i)}),\dotso \pi(L_{\sigma(n)}))$. However, $$ (\pi(L_1), \dotso \pi(L_i),\dotso \pi(L_n))=(\pi(L_{\sigma(1)}, \dotso \pi(L_{\sigma(i)}),\dotso \pi(L_{\sigma(n)}))$$ in $\overline{M}^{(n)}$.   That means, they are equal in $\overline{\mathcal{M}}$. Therefore, the map $\pi$ is well defined. 

  We now define the topology on $\overline{F}$ to be the \emph{initial topology induced by $\pi$}: that is, a subset $U \subseteq \overline{F}$ is open if and only if there exists an open set $V \subseteq   \overline{\mathcal{M}}$ such that $U={\pi}^{-1}(V)$. In particular, the preimage under $\pi$ of any open subset of $\overline{M}^{(n)}$, for every $n \in \mathbb{N}$,  defines an open subset of $\overline{F}$.

  Moreover, We can equip $F^{ab}$ with a topology in which $\mathcal{O}$ is declared as an open singleton set, and the topology on $\overline{F} \subset F^{ab}$ is defined as above.

  \subsection*{Proper Discontinuity of the Action of \texorpdfstring{$\rho(\Gamma)$}{ρ(Γ)} on \texorpdfstring{$\overline{F}$}{F}.}
  We know that $\rho(\Gamma)$ acts on $M$.
  Define the action of $\rho(\Gamma)$ on $M^{(n)}$ as follows:
$$ g \cdot (m_1, m_2, \dotso m_n)=(g(m_1), g(m_2), \dotso g(m_n)),$$  for all $g \in \rho(\Gamma),\,\,(m_1, m_2, \dotso m_n) \in M^{(n)} $. Moreover, $$(g(m_1),  \dotso g(m_n))=(g(m_{\sigma(1)}),  \dotso g(m_{\sigma(n)})) \,\, \forall \, \sigma \in S_n,$$ in $\overline{M}^{(n)}$. \medskip
Therefore,  $\rho(\Gamma)$ acts on  $\overline{M}^{(n)}$
as well. This happens for all $n \in \mathbb{N}$. Hence, $\rho(\Gamma)$ acts on $\mathcal{M}$ and  $\overline{\mathcal{M}}$, respectively.

For each $n \in \mathbb{N}$, define a map $$p_n: M^{(n)} \longrightarrow M \,\, \hbox{such that}\,\,p_n(m_1, m_2, \dotso m_n)=m_1. $$
Clearly, each $p_n$ is continuous. Also, note that, each \,$p_n$ is $\rho(\Gamma)$-equivariant. Now, consider a map $$p: \mathcal{M} \longrightarrow M\,\, \hbox{such that}\,\, p\Big|_{ M^{(n)}}=p_n. $$
 Clearly, $p$ is a $\rho(\Gamma)$-equivariant continuous map. Since, the action of $\rho(\Gamma)$ is properly discontinuous on $M$, it follows from \lemref{ml} that, the action of $\rho(\Gamma) $ is also properly discontinuous on $\mathcal{M}$.

 Now, observe that, $\overline{M}^{(n)}={M}^{(n)}/{S_n}$ for every $n \in \mathbb{N}$.  Moreover,  our domain of discontinuity $M := \Omega_\rho $ is a  manifold, that means it is Hausdorff. This implies $M^{(n)}$ is Hausdorff as well. 
 
 It is well known that for a finite group acting continuously on Hausdorff space, the natural orbit map is proper. Hence, it can be said that the quotient map  $q_n:M^{(n) } \to \overline{M}^{(n)} $ is proper. Let us consider, $$q: \mathcal{M} \longrightarrow \overline{\mathcal{M}}\,\, \hbox{such that}\,\, q\Big|_{ M^{(n)}} :=q_n.$$
 
  Now, it follows from \lemref{ml2} that, $q$ is also a proper map. We have already shown that, $\rho(\Gamma)$ acts on ${M}^{(n)}$ as well as on $\overline{M}^{(n)}$. It is easy to see that, the quotient map $q_n$ is $\rho(\Gamma)$- equivariant, surjective  for all  $n\in \mathbb{N}$. Therefore, $q$ is $\rho(\Gamma)$- equivariant, proper, surjective map.

It is already shown that, the action of $\rho(\Gamma)$ is properly discontinuous on $\mathcal{M}$.   Hence, it follows from \lemref{ml3} that,  $\rho(\Gamma)$ acts properly discontinuously on  $\overline{\mathcal{M}}$.

    The group $\rho(\Gamma)$ acts on $F^{ab}$ by:
\[
g'\cdot(\mathcal{O}):=\mathcal{O}, \quad g' \cdot (\mathcal{E}^g \otimes \cdots  (\mathcal{E}^{h})^* \otimes \cdots \otimes \mathcal{E}^k ) := \mathcal{E}^{g'g} \otimes \cdots  (\mathcal{E}^{g'h})^*\otimes  \cdots \otimes \mathcal{E}^{g'k}.
\]
This action is well defined and respect the group structure. In fact, $\rho(\Gamma)$ acts as group automorphisms on $F^{ab}$. This action restricts on $\overline{F}$ also.

 Moreover, the continuous map $\pi:\overline{F}  \to \overline{\mathcal{M}}$ is $\rho(\Gamma)$-equivariant:
 $$\pi (g' \cdot (\mathcal{E}^g \otimes \cdots  (\mathcal{E}^{h})^* \otimes   \cdots \otimes \mathcal{E}^k ))=  \pi(\mathcal{E}^{g'g} \otimes \cdots  (\mathcal{E}^{g'h})^*\otimes  \cdots \otimes \mathcal{E}^{g'k})$$ $$=(g'g (m), \dotso g'h (m),\dotso g'k (m))=g' \cdot (g (m), \dotso h (m),\dotso k (m))$$ $$=g' \cdot \pi ((\mathcal{E}^g \otimes \cdots  (\mathcal{E}^{h})^*  \cdots \otimes \mathcal{E}^k )), \,\, \forall \,g' \in \rho(\Gamma).$$

It is already proved that, the action of $\rho(\Gamma)$ on $\overline{\mathcal{M}}$ is properly discontinuous. Applying \lemref{ml} (the equivariant proper discontinuity lemma), we conclude that the action of $\rho(\Gamma)$ on $\overline{F}$ is properly discontinuous.
\begin{remark}
The construction of the free abelian group $F^{ab}$ depends genuinely on the
choice of the domain of discontinuity $\Omega_\rho$.  
Indeed, the generators of $F^{ab}$ are the classes of the pullback
bundles $\mathcal{E}_g=g^{*}T^{1,0}\Omega_\rho$. 
If $\Omega_\rho$ and $\widetilde{\Omega}_\rho$ are two admissible domains
of discontinuity, the induced actions of
$\rho(\Gamma)$ on these domains may identify different subsets of
elements under the equivalence relation $\sim$  defined in \S\ref{5.2} on $\rho(\Gamma)$.  In particular,
\[
g|_{\widetilde{\Omega}_\rho}=h|_{\widetilde{\Omega}_\rho}
\quad \nRightarrow \quad
g|_{\Omega_\rho}=h|_{\Omega_\rho},
\]
so the corresponding pullback bundles $\mathcal{E}^g$ and $\mathcal{E}^h$ may coincide for
one choice of domain but not for another. Therefore, there may be no  bijection between the free generators of  $F^{ab}(\Omega_\rho)$ and $F^{ab}(\widetilde\Omega_\rho)$.   In this sense,
$F^{ab}$ depends on the choice of the domain of discontinuity. The following statement also follows from this construction. 

Let $\rho:\Gamma\to G$ be a $(P^+,P^-)$--Anosov representation and let
$\Omega,~\widetilde\Omega$ be two $\rho(\Gamma)$--invariant domains of
discontinuity with $\widetilde\Omega\subset\Omega$. Write $F^{ab}(\Omega)$ and $F^{ab}(\widetilde\Omega)$ for the free abelian 
groups constructed from the corresponding pullback bundles on $\Omega$ and
$\widetilde\Omega$.  Assume the inclusion $\iota:\widetilde\Omega\hookrightarrow\Omega$
is $\rho(\Gamma)$--equivariant (in particular it preserves the action).
Then the assignment
$$
[\mathcal{E}^g]_{\Omega}\longmapsto [\mathcal{E}^g]_{\widetilde\Omega}\qquad (g\in\rho(\Gamma))
$$
on generators extends to a canonical surjective group homomorphism
$$
\phi:\; F^{ab}(\Omega)\longrightarrow F^{ab}(\widetilde\Omega).
$$
\end{remark}

\subsection{Category theoretic aspects}\label{subsec2.6}
Now, we will give a structure of category on  the $(P^+,P^-)$-Anosov representations of $\Gamma$ into $G$. We will  denote this category by $\textbf{An}$, whose class of objects (denoted by $ob(\textbf{An})$) consists of $(P^+,P^-)$-Anosov representations. Note that, the entire construction of this category depends on the choice of domain of discontinuity $\Omega_\rho$, for every     $\rho$. 

Let us denote, the category of  abelian groups (which is induced as a subcategory of groups) as $\textbf{Ab}$.
 Now, consider a map $\mathcal{F}: ob(\textbf{An}) \longrightarrow ob(\textbf{Ab})$ such that  $\mathcal{F}(\rho)= F_{\rho}^{ab}$ for every $(P^+,P^-)$-Anosov representation $\rho$. Here, $F_{\rho}^{ab}$ is the free abelian group constructed in \S\ref{ab}. When, $\Omega_\rho=\emptyset$ for any  $\rho$, it will be assumed that $\rho(\Gamma)$ acts trivially on $\emptyset$. In that case, $F_{\rho}^{ab}$  will be  the identity group. Note that, while considering the map $\mathcal{F}$, our choice of domains of disontinuity must be the complex curves.
Now, see the proof of \thmref{thm1.3}.
 \subsubsection{Proof of \thmref{thm1.3} }
For any two representations $\rho$ and $\rho'$, define the set of morphisms $Hom(\rho, \rho')$ as follows. 
$$
\mathrm{Hom}(\rho, \rho') := \left\{ \phi \in \mathrm{Hom}(G, G) \,\middle|\,
\begin{aligned}
&\phi(\rho(\Gamma)) \subseteq \rho'(\Gamma), \\
&\text{and } \forall\, g, h \in \rho(\Gamma),\, g|_{\Omega_\rho} = h|_{\Omega_\rho} 
\Rightarrow \phi(g)|_{\Omega_{\rho'}} = \phi(h)|_{\Omega_{\rho'}}
\end{aligned}
\right\}.
$$
It can be easily seen that the trivial homomorphism  $(\phi(G)=e)$ always satisfies the above criteria. Therefore,
the set $Hom(\rho, \rho')$ is always nonempty for any two representations $\rho$ and $\rho'$.

Let \(\phi \in \mathrm{Hom}(\rho, \rho')\) and \(\psi \in \mathrm{Hom}(\rho', \rho'')\). Then
$$
\psi \circ \phi(\rho(\Gamma)) \subseteq \psi(\rho'(\Gamma)) \subseteq \rho''(\Gamma).
$$
Moreover, if \(g, h \in \rho(\Gamma)\) satisfy \(g|_{\Omega_\rho} = h|_{\Omega_\rho}\), then by definition of \(\phi\), we have \(\phi(g)|_{\Omega_{\rho'}} = \phi(h)|_{\Omega_{\rho'}}\). Since \(\phi(g), \phi(h) \in \rho'(\Gamma)\), the definition of \(\psi\) implies
$$
\psi(\phi(g))|_{\Omega_{\rho''}} = \psi(\phi(h))|_{\Omega_{\rho''}}.
$$
Hence, \(\psi \circ \phi \in \mathrm{Hom}(\rho, \rho'')\), and composition is closed in this morphism set.

 Note that, the identity  homomorphism of $G$, $id \in Hom(\rho, \rho)$, for any $\rho \in ob(\textbf{An})$, satisfies $\phi \circ id=\phi$ and $id \circ \psi=\psi$, for any $\phi\in Hom(\rho, \rho')$ and $\psi\in Hom(\rho'', \rho)$. Also, the associativity property of these morphisms follows easily.

 In this way, we  can give a structure of category on the space of  $(P^+,P^-)$-Anosov representations with the defined set of  morphisms and their composition.

For any \(\phi \in \mathrm{Hom}(\rho, \rho')\) with \(\Omega_\rho, \Omega_{\rho'} \neq \emptyset\), define a map
$$
\mathcal{F}(\phi): F_{\rho}^{ab} \longrightarrow F_{\rho'}^{ab} \quad \text{by} \quad \mathcal{F}(\phi)(\mathcal{E}_\rho^h) := \mathcal{E}_{\rho'}^{\phi(h)} \quad \text{for all } h \in \rho(\Gamma).
$$
This map is well-defined: if \(\mathcal{E}_\rho^g = \mathcal{E}_\rho^h\), then \(g|_{\Omega_\rho} = h|_{\Omega_\rho}\), and since \(\phi \in \mathrm{Hom}(\rho, \rho')\), it follows that \(\phi(g)|_{\Omega_{\rho'}} = \phi(h)|_{\Omega_{\rho'}}\), hence \(\mathcal{E}_{\rho'}^{\phi(g)} = \mathcal{E}_{\rho'}^{\phi(h)}\).

Thus, \(\mathcal{F}(\phi)\) is a well-defined group homomorphism. If either \(\Omega_\rho\) or \(\Omega_{\rho'}\) is empty, we interpret \(\mathcal{F}(\phi)\) as the trivial homomorphism.

It is clear that \(\mathcal{F}(\mathrm{id}) = \mathrm{id}_{F_{\rho}^{ab}}\), and for any \(\phi \in \mathrm{Hom}(\rho, \rho')\), \(\psi \in \mathrm{Hom}(\rho', \rho'')\), we have
$$
\mathcal{F}(\psi \circ \phi) = \mathcal{F}(\psi) \circ \mathcal{F}(\phi).
$$
Therefore, \(\mathcal{F}: \mathbf{An} \longrightarrow \mathbf{Ab}\) defines a covariant functor. \qed

\section{ On Character Variety}\label{subsec2.5} 
It is already discussed about the construction of  $\Omega_\rho^I$  in \S\ref{dod}. In this section, we consider this domain of continuity specifically for our construction.  Now, consider the case, when $M := \Omega_\rho^I$ is a complex curve. We will prove the following theorem for this domain of discontinuity.
\begin{theorem}\label{t2.11}
    The free abelian  group of the corresponding representation is well defined on the character variety of the Zariski dense $P$-Anosov representations upto isomorphism.  
\end{theorem}
\begin{proof}
    Let, $\rho$ be a Zariski dense  $P$-Anosov representation. From   \cite[Proposition 3.3]{gn}, it follows that $g\rho g^{-1}$ is also a Zariski dense $P$-Anosov representation with associated limit map $g\xi$. Then $$\Omega_{g\rho g^{-1}}^I = \mathcal{X} \smallsetminus \bigcup_{z\in {\partial_{\infty} \Gamma}}\{x\in \mathcal{X}: \hbox{pos}(g\xi_\rho(z),x)\in I)\}=\mathcal{X} \smallsetminus \bigcup_{z\in {\partial_{\infty} \Gamma}}\{gx\in \mathcal{X}: \hbox{pos}(\xi_\rho(z),x)\in I\}.$$
    The second equality follows because  $\hbox{pos}(\xi_\rho(z),x)=\hbox{orbit of}~ (\xi_\rho(z),x)$ under the diagonal action of $G$, which is the same as $\hbox{pos}(g\xi_\rho(z), gx)$ for all $g\in G$. From here, it can be  seen easily that $\Omega_{g\rho g^{-1}}^I =g\Omega_\rho^I$.

Denote the free abelian group (constructed in \S\ref{ab}) associated to the Anosov representation $\rho$ by $F_{\rho}^{ab}$, with generators $\{\mathcal{E}_\rho^h\}_{[h] \in \overline{\rho(\Gamma)}}$. Define a group homomorphism 
$$
\psi : F_{\rho}^{ab} \longrightarrow F_{g \rho g^{-1}}^{ab}, \quad \psi(\mathcal{E}_\rho^h) = \mathcal{E}_{g \rho g^{-1}}^{g h g^{-1}}, \quad \forall\, h \in \rho(\Gamma).
$$

Suppose $\mathcal{E}_\rho^h = \mathcal{E}_\rho^{h'}$ for some $h, h' \in \rho(\Gamma)$, i.e., $h m = h' m$ for all $m \in \Omega_\rho^I$. Note that any element of $\Omega_{g \rho g^{-1}}^I$ is of the form $g m$ for some $m \in M$, and  
$$
( g h g^{-1} ) g m = g h m = g h' m = ( g h' g^{-1} ) g m, \quad \forall m \in M.
$$
Hence, $\mathcal{E}_{g \rho g^{-1}}^{g h g^{-1}} = \mathcal{E}_{g \rho g^{-1}}^{g h' g^{-1}}$, so $\psi$ is well-defined.

Moreover, $\psi$ induces a bijection between the generators of $F_{\rho}^{ab}$ and $F_{g \rho g^{-1}}^{ab}$, and therefore  
$$
F_{\rho}^{ab} \cong F_{g \rho g^{-1}}^{ab}.
$$
This establishes the desired result.
\end{proof}

\begin{remark}
%An insightful  observation arises  when we choose the domain of discontinuity as $\Omega_\rho^I$.
Suppose \(\rho: \Gamma \to G\) is a Zariski dense \(P\)-Anosov representation, and let \(\rho' := g \rho g^{-1}\) be its conjugate by some \(g \in G\). Then, by the above construction,  the inner automorphism \(\phi: G \to G\) defined by \(\phi(h) = ghg^{-1}\) for all \(h \in G\) belongs to \(\mathrm{Hom}(\rho, \rho')\). Moreover, \(\mathcal{F}(\phi): F_{\rho}^{ab} \to F_{\rho'}^{ab} \) is a group isomorphism.  

This demonstrates how inner automorphisms of 
$G$ preserve the equivalence classes of line bundles on the associated domain $\Omega_\rho^I$, thus ensuring that the morphisms in the category $\textbf{An}$ appropriately preserve the geometric structure encoded by the associated free abelian group $F_{\rho}^{ab}$.
\end{remark}

\end{document}